\documentclass[11pt]{amsart}
\usepackage {amsmath, amssymb,   epsfig, a4wide, enumerate, psfrag}
\usepackage[latin1]{inputenc}
\usepackage[english, frenchb]{babel}

\date{\today}

\author{Romain Dujardin}
\thanks{Research  partially supported by ANR project LAMBDA,  ANR-13-BS01-0002.
}
 
\title{A non-laminar dynamical Green current}

\address{LAMA  \\
Universit\'e Paris-Est Marne-la-Vall\'ee  \\
5 boulevard Descartes \\
77454 Champs sur Marne  \\
France}
\email{romain.dujardin@u-pem.fr}


\newcommand{\cc}{\mathbb{C}}

\newcommand{\pp}{\mathbb{P}}

\newcommand{\cv}{\rightarrow}

\newcommand{\fr}{\partial}
\newcommand{\om}{\Omega}
\newcommand{\set}[1]{\left\{#1\right\}}
\newcommand{\norm}[1]{\left\Vert#1\right\Vert}
\newcommand{\abs}[1]{\left\vert#1\right\vert}
\newcommand{\cd}{{\cc^2}}
\newcommand{\pd}{{\mathbb{P}^2}}

\newcommand{\pk}{{\mathbb{P}^k}}

\newcommand{\rest}[1]{ \arrowvert_{#1}}

\newcommand{\unsur}[1]{\frac{1}{#1}}
\newcommand{\el}{\mathcal{L}}

\newcommand{\lrpar}[1]{\left(#1\right)}

\newcommand{\itm}{\item[-]}

\DeclareMathOperator{\supp}{Supp}

\DeclareMathOperator{\dist}{dist}

\newtheorem{prop}{Proposition} [section]
\newtheorem{thm}[prop] {Theorem}

\newtheorem{lem}[prop] {Lemma}
\newtheorem{cor}[prop]{Corollary}

\theoremstyle{remark}

\begin{document}
\selectlanguage{english}

\begin{abstract} 
A holomorphic endomorphism $f$
of $\mathbb{CP}^2$ admits a Julia set $J_1$, defined as usual to be  the locus of non-normality of its iterates $(f^n)_{n\geq 0}$, 
and a (typically) smaller Julia set $J_2$, which is essentially the closure of the set of repelling periodic orbits. 
The question has been raised whether 
$J_1\setminus J_2$ is filled   (possibly in a measure-theoretic sense) with 
 ``Fatou subvarieties" along which the dynamics is locally equicontinuous. 
In this article we construct examples showing that this is not the case in general. 

\end{abstract}

\maketitle

\section{Introduction}

\subsection{Background}
Let $f$ be a holomorphic endomorphism of $\mathbb{CP}^2$, of degree $d\geq 2$. 
The Julia set $J_1$ (or simply  $J$) of $f$ is classically defined as the locus where the iterates $(f^n)_{n\geq 0}$ do not locally form a normal family. 
Contrary  to the one-dimensional case, the closure of the set of repelling periodic orbits is typically smaller than $J_1$. 
For instance, for the endomorphisms induced by polynomial mappings on $\cd$,   
   the Julia set is unbounded in $\cd$ while repelling periodic orbits stay in a compact subset. 
 
 Let $J_2$ be the support of the so-called 
 {\em equilibrium measure} of $f$, which is the unique measure of maximal entropy.
 Repelling periodic orbits are dense in  $J_2$ (see Briend-Duval \cite{briend duval}). 
 It turns out that $J_2$ is   better  behaved as an analogue of the classical Julia set than the 
   closure of repelling orbits, which may have isolated points (see Hubbard-Papadopol \cite[p. 345]{hp}).     
 
 Let  $T$   be the dynamical Green current  of $f$, defined by $T = \lim_{n\cv\infty} d^{-n} [f^{-n}(L)]$, 
 where $L\subset \pd$ is a generic line. 
The Julia set  $J_1$ coincides with  $\supp(T)$, and the self-intersection measure
$ T\wedge T$ is the  measure of maximal entropy of $f$ (see Forn\ae ss-Sibony \cite{fs2}).

 \medskip
 
 It has been a long standing problem in higher dimensional holomorphic dynamics to describe the structure of $J_1\setminus J_2$.
  A popular picture is that $J_1\setminus J_2$ should be ``foliated" (in some appropriate sense) by holomorphic disks $D$ along which 
  $(f^n\rest{D})_{n\geq 0}$ is a normal family. This happens to be the case in all the examples which have been analyzed so far. 
  Such disks will be referred to as {\em Fatou disks}. For instance, 
  in \cite{fs2}, Forn\ae ss and Sibony 
  define a set $J'_2$ (actually denoted by $J_1$ in their paper) as follows: $x$ belongs to $(J'_2)^c$ 
  if there exists a neighborhood $N\ni x$ such that for every $y\in N$ there is a germ of Fatou disk through $y$.   
  They show that $J_2\subset J'_2$ and ask whether   equality holds. 
 
 \medskip
 
In \cite{fatou}, we gave the following  picture of the infinitesimal dynamics on $J_1\setminus J_2$:
for $\sigma_T$-a.e. $x\in J_1\setminus J_2$ ($\sigma_T$ is the trace measure of $T$), 
there exists a {\em Fatou direction} $\mathcal{F}_x$ in 
$T_x\pd$ along which $df^n$ is not expanding, while $df^n$ expands exponentially in the remaining 
 directions. The question of the integrability of this 
  field of Fatou directions   into a field of Fatou disks was left open in that paper. 

 \medskip

A related question is whether the Green current $T$ is {\em laminar} in $J_1\setminus J_2$. Recall that a positive current in 
$\om\subset \cd$ is  {laminar} if it expresses as an integral of integration currents over a measurable  family of compatible holomorphic 
disks (compatible here means that these disks have no isolated intersections). These disks are then automatically of 
Fatou type \cite{fs2}.
The reader is referred to  \cite{bls, structure} for basics on laminar currents. 
In his thesis De Thélin came quite close to a positive answer to this question (see \cite{dt-lamin, dt-genre}). First, 
he showed that $T$ is laminar on $J_1\setminus J_2$ when $f$ is post-critically finite. 
In the general case, he proved the following two results:
\begin{itemize}
\itm If $(C_n)$ is a sequence of complex submanifolds (i.e. curves)  
in some open set $\om\subset \cd$ such that $\mathrm{Area}(C_n)  = d^n$ and $\mathrm{genus}(C_n) = O(d^n)$, then any cluster limit of the sequence of currents $d^{-n}[C_n]$ is laminar in $\om$. 
\itm Under a generic condition on $f$ and $L$, if $\om\subset \pd$ is an open set such that $\overline \om$ is disjoint from $J_2$, then 
\begin{equation}\label{eq:genre}
\limsup_{n\cv\infty}\unsur{n}\log \mathrm{genus}(f^{-n}(L)\cap \om) \leq \log d
\end{equation}  (recall that $f^{-n}(L)$ is smooth for generic $L$ and that 
 $\mathrm{Area}(f^{-n}(L)) = d^n$).
 \end{itemize} 
Altogether, this  appears  as   strong  evidence in favor of  the laminarity of $T$ outside $J_2$. 

\medskip

The main result in  this paper is that, somewhat    surprisingly,  the answer to the laminarity problem is ``no" in general. We construct explicit 
examples of endomorphisms of $\pd$ (actually polynomial on $\cd$) such that the Green current is {\em not} laminar outside $J_2$, even in a 
very weak sense. In particular this shows that the estimate in  \eqref{eq:genre} cannot be improved to   $O(d^n)$.
For these examples, we also have that $J_2\neq J'_2$, thereby solving  by the negative
the question of Forn\ae ss and Sibony.

\subsection{Setting and main result}
Let $f(z,w) = (p(z), q(z,w))$ be a polynomial skew product in $\cd$  
It is convenient to view the second component as a non-autonomous dynamical system on $\cc$, denoted by $q_z(\cdot)$, so that 
$$f^n(z,w)  = (p^n(z), q_{z_{n-1}}\circ    \cdots \circ   q_{ z} (w)), \text{ where } z_j = p^j(z).$$ We let 
$q^n _z  =  q_{z_{n-1}}\circ    \cdots \circ   q_{ z} $. The reader is referred to Jonsson \cite{jonsson} for basics  on  the 
 dynamics of polynomial skew products in $\cd$  (see also Sester \cite{sester}). 

We assume that 
$p(z) = z^d+ {\rm l.o.t.}$ and $q(z,w) = w^d+ {\rm l.o.t.}$,  so that $f$ extends to a holomorphic self-mapping on $\pd$. Such polynomial mappings will be referred to as {\em regular}. 
We note that in this setting, the Green current was shown to be laminar in the basin of the line at infinity by Bedford and Jonsson \cite{bj}. 
Our methods require that $d\geq 3$, 
conjecturally one should be able to treat the case $d=2$ using similar  ideas.  We assume that 
$p$ admits a linearizable irrational fixed point at $z=0$, that is,  $p(z) = e^{2\pi i\theta} z+ \cdots + z^d$, where the rotation number 
$\theta$ satisfies the Brjuno condition ($\theta = (\sqrt5-1)/2$ would do). Then $p$ 
admits a Siegel disk $\Delta$ centered at the origin.   

Since $f$ is polynomial on $\cd$    its
 dynamical Green function is defined by   $$G(z,w)  = \lim_{n\cv\infty}   {d^{-n}}\log^+ \norm{f^n(z,w)},$$ and the 
Green current $T$  satisfies $T=dd^cG$. From now on we simply  denote the Julia set by $J$ (recall that  $J = \supp(T)$). 
In every vertical fiber $\set{z}\times \cc$ 
we can define a vertical filled Julia set $K_z$ as the set of points $(z,w)$ where the sequence $(q_z^n(w))_{n\geq 1}$ is bounded. 
The vertical Julia set is by definition
$J_z= \fr K_z$, and it coincides with the locus of non-normality of $(q^n_z)_{n\geq 1}$. 
Notice that the sets $K_z$ are locally uniformly bounded. 
For polynomial skew-products, $J_2$ can be characterized as being
the closure of the set of repelling periodic orbits \cite[\S 4]{jonsson}, and $J_2 = \overline{\bigcup_{z\in J_p} J_z}$. 
Therefore in our situation,   $J\cap (\Delta \times \cc)$ is disjoint from $J_2$. In particular $T\wedge T =0$ in $\Delta\times \cc$. 

\medskip

The following set of assumptions on an endomorphism $f$ will be denoted by (A): 
\begin{itemize}
\item[(A1)] There  exists an invariant  hyperbolic  compact and connected subset $E_0\subset J(q_0)$, relative to $q_0$ (not reduced to a point).
\item[(A2)] There exists a critical point $c$ for $q_0$ and an integer $k$ such that $c\notin E_0$ and 
$q_0^k(c)$ is a repelling periodic point $m$ belonging to 
$E_0$.
\item[(A3)] There is a   local  component $V$ of the critical set 
$\mathrm{Crit }(f)$ at $(0,c)$ such that  $f^k(V)$ is smooth at $(0,m)$ and not periodic.
\end{itemize}

In Proposition \ref{prop:existence} below we show that such a mapping exists for every $d\geq 3$.

\medskip


 We say that a  closed  positive 
current in $\om\subset \cd$  is {\em quasi-laminar}
 if for $\sigma_T$-a.e. $x$, there exists a germ of holomorphic disk $D$ containing $x$ such 
that  $u\rest{D}$ is harmonic, where $u$ is any local potential for $T$. If $T$ is the Green current for an endomorphism of $\pd$, 
the harmonicity of its  potential $G$ 
 on $D$ is equivalent to the equicontinuity of $f^n$ along $D$ \cite{fs2}. Therefore, the quasi-laminarity of 
$T$ in some open set means that the field of Fatou directions is integrable  (in a measure-theoretic sense) 
to a field of Fatou disks.

 Here is the precise statement  of our main theorem.

\begin{thm}\label{thm:non quasi laminar}
Let $f$ be a regular polynomial skew product in $\cd$ of the form 
$f(z,w)  = (p(z), q_z(w))$,  where $p(z) = e^{2\pi i\theta} z+ \cdots +z^d$ has a linearizable fixed point at the origin, and  
satisfying  the above    assumptions (A).
Then there exists a neighborhood $N$ of 0 such that 
the  Green current $T$ is not   quasi-laminar  in any  open subset of $N\times \cc$.
\end{thm}

%

Equivalently, if $\om\subset N\times \cc$ is any open subset, there exists a set $E\subset \om$ of positive trace measure such that 
for every $x\in E$, there us no Fatou disk through $x$.  

The following corollary could easily be deduced from the theorem together with  the fact that $T\wedge T = 0$ in $\Delta\times \cc$. 
We will actually give a direct proof. 

\begin{cor}\label{cor:non laminar}
Let $f$  and $N$ 
be as in Theorem \ref{thm:non quasi laminar}.  The Green current $T$ is not   laminar  in    any  open subset of $N\times \cc$.
\end{cor}

The   first (non-dynamical) examples of non-laminar 
positive closed currents in $\cd$ 
with continuous potential and vanishing  self-intersection were constructed in \cite{wermer}.
Since $T\wedge T=0$  in $\Delta\times \cc$, 
$T$ can thus be considered as a kind of dynamical version of these examples. Notice 
however that, as opposed to the ``Wermer examples" of \cite{wermer},  
the Julia sets that we construct  contain many holomorphic disks. 

\medskip

Recall that the set $J'_2$ was defined to be  the set of points 
 such that  locally there does not exists a neighborhood   that is filled with germs of  Fatou disks. 
 In the particular context of polynomial skew products in $\cd$,  the question whether $J_2$ equals $J'_2$ is raised in \cite[\S 8]{jonsson}.
  The following corollary is an immediate consequence of Theorem \ref{thm:non quasi laminar}. 

\begin{cor}\label{cor:J'_2}
Let $f$ and $N$ be as in Theorem \ref{thm:non quasi laminar}.    Then $J\cap (N\times \cc)$ is contained in $J'_2$, so in particular $J_2\subsetneq J'_2$. 
\end{cor}

Indeed, if $\om\subset N\times \cc$ was an open set intersecting $J$ with the property
 that through every point in $\om$ there exists 
 a Fatou disk, then $T$ would be quasi-laminar in $\om$, a contradiction.
 
 \medskip
 
 Theorem \ref{thm:non quasi laminar} and Corollary \ref{cor:non laminar} will be established in \S \ref{sec:proofs}. In  \S \ref{sec:further}, we discuss the higher dimensional case, and comment about a question raised by Forn\ae ss and Sibony in \cite{fs questions}.

\section{Proof of the main results}\label{sec:proofs}

\subsection{Examples} We start by exhibiting   examples satisfying the assumptions (A)  of the main theorem.

\begin{prop}\label{prop:existence}
Let $d\geq 3$ and fix $p(z)  = e^{2\pi i\theta} z+ \cdots + z^d$ as in Theorem \ref{thm:non quasi laminar}. Then there
   exists a polynomial $q(z,w)$ such that  the skew product $f(z,w)  = (p(z), q_z(w))$ satisfies (A).
   \end{prop}

\begin{proof}
To carry out the few computations to come, we need to find $q_0$ in a rather explicit form. For this it is convenient 
to use the work of P. Roesch \cite{roesch}. 
Assume 
that $g(w)$ is a polynomial of degree $d$ with a super-attracting point of multiplicity at least $d-1$ at the origin, and denote its immediate
basin by 
 $B(0)$. We denote by $\mathcal S$ the parameter space of 
   such mappings. Since polynomials can be parameterized by specifying their critical points, 
    $\mathcal{S}$ is one-dimensional, and parameterized by the remaining critical point (this differs from the presentation in \cite{roesch}).  
In this situation, 
 $\fr B(0)$ is a Jordan curve (we will not use this fact) and it may happen that the  free   critical point $c$ satisfies $c\notin 
\overline{B(0)}$ and $g(c)\in \fr B(0)$. Indeed consider the set $\mathcal{H}\subset \mathcal S$ of such mappings $g$ such that $g^n(c)$ 
converges to 0. Then $\mathcal H  = \bigcup_{i\geq 0} \mathcal H_i$, where $g\in \mathcal H_i$ if and only if $i$ is the first integer such 
that $g^i(c)\in B(0)$. For every $i$, $\mathcal{H}_i$ is non-empty.
 Roesch shows that if $i>0$, if $\om$ is a component of $\mathcal{H}_i$, and if  $g\in \fr\om$, then $c\notin 
\overline{ B(0)}$ and $g^i(c)\in \fr B(0)$. In particular it is clear that  $g$ admits no  parabolic periodic points neither any 
 recurrent critical 
point.    Hence from Ma\~né's Theorem \cite{mane}, the invariant compact set $\fr B(0)$ is hyperbolic, in particular it is locally persistent and repelling periodic points are dense there. Since $g^i(c)$ does not persistently belong to $\fr  B(0)$, by
 perturbing $g$ within   $\mathcal S$, we may further assume that $g^i(c)$ is a repelling periodic point.  We fix $q_0$ to be such a $g$, with, say, $i=1$, which thus satisfies (A1) and (A2). 

\medskip

It remains to check the third assumption (A3) that the component of the critical set  containing $(0, q_0(c))$ in $\cd$ is not periodic. 
In order to use some results from   \cite{preper}, 
we introduce the  parameterization of $\mathcal{S}$  given by 
$$\mathcal{S} = \set{\frac{w^d}{d} - c \frac{w^{d-1}}{d-1}, c\in \cc},$$
where $c$ is both the parameter and the free critical point. We choose $q_z(w)$ of the form 
$$q_z(w) = \frac{w^d}{d} - c \frac{w^{d-1}}{d-1} + \beta z,$$ where  $c$ is fixed so that $q_0$ is as   above, and 
$\beta\in \cc$ is a free parameter (of course then we can go back to the initial normalization 
 $q_z(w) = w^d+ \mathrm{l.o.t.}$ by  a linear conjugacy).
Let us show that   (A3) is fulfilled for generic $\beta$. 

The critical set of $f$ is  
\begin{align*}
\mathrm{Crit}(f)  &= \lrpar{ \mathrm{Crit}(p) \times \cc}  \cup \set{(z,w), \ \frac{\fr q}{\fr w} (z,w) = 0}\\  &= 
\lrpar{ \mathrm{Crit}(p) \times \cc}  \cup (\cc\times \set{0})\cup (\cc\times \set{c}),
\end{align*}
and we need to prove that $\cc\times \set{c}$ is not preperiodic for generic $\beta$. Notice that  $f^i(\Delta\times \cc)$ is smooth, 
because the image of a graph over $\Delta$ is a graph (see the proof of Lemma \ref{lem:motion}   for some details on this).

Let $z_1$ be a fixed point of $p$, distinct from 0. We work in the fixed fiber $\set{z_1}\times \cc$.
It is enough to prove that $c$ is not persistently preperiodic in this fiber. For this, we show that it escapes for large $\beta$.
To stick with the notation  of  \cite{preper},    put  $\beta z_1= a^d$. Then the restriction of $f$ to $\set{z_1}\times \cc$ becomes 
 $P_{c,a} = \frac{w^d}{d} - c \frac{w^{d-1}}{d-1} +a^d$.
We have that 
$P_{c,a}(0) = a^d$ and $P_{c,a}(c) = -\frac{c^d}{d(d-1)} + a^d$. By \cite[Prop. 6.3]{preper}, since $c$ is fixed,
$$G(c,a ) = \max \left(g_{ P_{c,a}}(0), g_{ P_{c,a}}(c)\right) = \log^+\max (\abs{a}, \abs{c}) + O(1) = \log^+\abs{a} +O(1)$$
($g_{ P_{c,a}}$ is the dynamical Green function of $P_{c,a}$).  Now by 
 \cite[Lemma 6.5]{preper}, for every $z\in \cc$ we have that 
 $$\max( g_{P_{c,a}}(z), G(c,a )) \geq \log \abs{z-\frac{c}{d-1}} -\log 4.$$ Applying this to $z = P_{c,a}(c) = a^d+O(1)$, we see that 
$$\max( g_{P_{c,a}}(P_{c,a}(c)), G(c,a ))\geq d\log \abs{a} + O(1) \text{ as } a\cv\infty.$$ Thus   when $\abs{a}$ is large, 
$g_{P_{c,a}}(P_{c,a}(c)) > G(c,a )$ and we deduce that $P_{c,a}(c)$, hence $c$, escapes, which was the desired result. 
The proof is complete.\end{proof}

\subsection{Proofs} We now prove Theorem \ref{thm:non quasi laminar} and Corollary \ref{cor:non laminar}. 
Let $f$ be as in the statement of the theorem. Replacing it by an iterate we may assume that the 
integer $k$ in (A) equals 1 and that $(0,m)$ is fixed. By convention, we only consider neighborhoods $N\times \cc$ of the central fiber such 
that $N$ is invariant under $p$, that is, $N$ is a disk in the linearizing coordinate. Throughout the proof we freely reduce $N$ without changing 
its notation. We denote by $\pi_1$  the first coordinate projection in $\cd$. 

\medskip

{\noindent \em Step 1: geometry of the Julia set.} 
We first  study the persistence of $E_0$ in the fibers close to 
$\set{0}\times \cc$. 

\begin{lem}\label{lem:motion}
There exists a neighborhood $N$ of 0 and a  $f$-equivariant    holomorphic motion $E$ of $E_0$ over $N$, that is, if $\gamma_w$ denotes the graph of the motion through $(0,w)$, we have $f(\gamma_w)  = \gamma_{q_0(w)}$. 

Furthermore, $f$ is vertically expanding along $E$, that is, there exists   $C>0$ and $\beta >1$ such that  if $(z,w)\in E$,
$\norm{df^n_{(z,w)}(0,1)}\geq C\beta^n$.  
\end{lem}

Notice that $E_0$ is not a hyperbolic set for $f$ in $\cd$ so we cannot simply invoke hyperbolicity here. Still, the 
argument is based on standard ideas.

\begin{proof}
Since $E_0$ is a hyperbolic set for $q_0$, 
for every $w\in E$ there exists a disk $D_w$ centered at $w$ such that $q_0$ is invertible on $D_w$ and 
 if $q_{0,w}^{-1}$ denotes the    inverse branch of $q_0$ sending $q_0(w)$ to $w$, then $q_{0,w}^{-1} (D_{q_0(w)}) \Subset D_w$. 
Hence if $N$ is a sufficiently small $p$-invariant   neighborhood of $0$, for every $(0,w)\in \set{0}\times E_0$, $f$ is invertible in 
$N\times D_w$ and if as before $f_w$ is the corresponding inverse branch,  $f_w^{-1}( N\times D_{q_0(w)}) $ is horizontally contained in 
the ``topological bidisk" 
$N\times D_w$. Recall that this means that if $L$ is any vertical line in    $N\times D_w$, $f_w^{-1}( N\times D_{q_0(w)})\cap L \Subset L$. 
Thus the situation   is similar to that of a 
crossed mapping (see Hubbard and Oberste-Vorth \cite{ho}) with the difference that since $f$ 
preserves the vertical fibers,  it maps the vertical piece of the boundary of $f_w^{-1}( N\times D_{q_0(w)})$ to the vertical boundary of $N
\times D_{q_0(w)}$, that is, 
$$f\lrpar{ \overline{f_w^{-1}( N\times D_{q_0(w)})}\cap (\fr N\times D_w)}\subset \fr N\times  D_{q_0(w)}.$$ 
In addition, $f\rest{N\times D_w}$ has degree 1 in the sense that the image of a graph over $N$ is a graph over $N$.  
Indeed write this graph as $(t, \gamma(t))$, $t\in N$,  whose image is 
$(p(t), q_t(\gamma(t))$, and observe that $p:N\cv N$ is invertible.  

\medskip

Let now $w_0\in E_0$ and consider its orbit  $(w_n)$ under   $q_0$. The sequence 
$f_{w_0} ^{-1} \cdots f_{w_{n-1}}^{-1} (N\cap D_{w_n})$ defines 
 a nested sequence of horizontal topological bidisks in $N\times D_{w_0}$. Exactly as in \cite{ho},
 the contraction property of the vertical Poincaré metric shows that it converges to a graph $\gamma_{w_0}$. 
 It is   straightforward to check that the family of graphs $\gamma_w$ is the desired holomorphic motion. The vertical expansion statement follows as well.
\end{proof}

Let us now recall a few facts from the work of Jonsson \cite{jonsson2}. For every $z$, let
 $\mu_z = T\wedge [\set{z}\times \cc]$, which is a probability measure supported exactly on $J_z$. 
 Every invariant circle $C$ in the Siegel disk admits a unique 
$p$-invariant  probability measure   which we denote by $\lambda_C$. 
Then the measure $\mu_C = 
\int \mu_z\lambda_C(dz)$ is $f$-invariant and ergodic. Given such a family of circles and a probability measure 
$\Lambda$ on this family we can   integrate with respect to $C$ to get an invariant measure for $f$. Taking $\Lambda$ to be smooth, 
we see in particular that there exists   invariant probability measures absolutely continuous with respect to $\sigma_T$, more precisely, with 
respect to $T\wedge idz\wedge d\overline z$. 
Observe that by the slicing formula, $T\wedge idz\wedge d\overline z$ is an integral of $\mu_z$, so we infer that 
 the support of $T\wedge idz\wedge d\overline z$ equals $\overline {\bigcup_{z\in \cc} J_z}$, 
which  is smaller than $J$ in general  (this happens for instance for $f(z,w) = (z^2, w^2)$). This is not the case in our situation, as the following lemma shows. 

\begin{lem}\label{lem:julia}
In $N\times \cc$, the Julia set coincides with $ \overline {\bigcup\nolimits_{z_\in N } J_z}$.
\end{lem}

\begin{proof}
Let $x = (z,w) \in J\cap (N\times \cc)$. There exists an arbitrary small
 holomorphic disk $D$ through $x$ such that 
$f^n\rest{D}$ is not a normal family. 
We claim that for large $n$, $f^n(D)$ intersects $E$. This implies the lemma because as $f$ is vertically expanding along $E$, 
$E$ is contained in $ \bigcup_{z\in N} J_z$, and since moreover 
$ \bigcup_{z\in \cc}  J_z$ is totally invariant,   we conclude that $D\cap \bigcup_{z\in N}  J_z\neq \emptyset$.

To prove the claim, we use a Kobayashi hyperbolicity argument.  
Since   $f^n\rest{D}$ is not a normal family, by the Zalcmann-Brody lemma,  some subsequence 
can be  reparameterized to converge to a non constant entire curve $\phi:\cc\cv N\times \cc$, which must then be contained in a 
vertical line. 
Now we claim that there is a   pluriharmonic function $h$ in  $(N\times \cc)\setminus E$ that is bounded from below   and  tends to infinity as $w\cv\infty$. So if for all $n\geq 0$,  $f^n(D)$ avoids $E$, $h\circ \phi$ is   harmonic and bounded from below on $\cc$, 
hence constant, and we get a contradiction. 

To construct the function $h$, we do as follows. Since $E_0$ is a non-trivial continuum, it is not polar so it carries a probability measure $m_0$ with continuous potential. Let $m_z$ be the image of $m_0$ under  the holomorphic motion at time $z$. 
In \cite[\S 6]{structure}  it is shown that the function $$u :(z,w)\longmapsto  \int \log\abs{w-s} dm_z(s)$$ is continuous, psh, and $dd^cu  = \int [\gamma_w] dm_0(w)$.  Therefore it is enough to choose $h=u$ and we are done. 
\end{proof}

\medskip 

{\noindent \em Step 2: persistent intersection with the post-critical set}. Recall that by assumption there is a smooth 
component $W$ of $f(\mathrm{Crit}(f))$ passing through the fixed point $(0,m)\in E$, and that is not fixed. 
This means that it does not coincide with the continuation $\gamma_m$ of $(0,m)$. Recall also that since $\theta$ is linearizable, 
there is a foliation of  the (punctured) Siegel disk $\Delta$ by invariant circles.  

The next lemma is where we use the connectedness of $E_0$.

\begin{lem}\label{lem:connex}
Reducing $N$ if needed,  the following property holds: if $C\subset N$ is any $p$-invariant circle,   there is an intersection between $E$ and $W$ over $C$. 
\end{lem}

\begin{proof}
The simplest situation is when $W$ is locally a graph over the $z$ coordinate and it is transverse to $\gamma_m$. In this case, 
$W\cap E$ is locally homeomorphic to $E_0$ near $(0,m)$.   The curves $\pi_1^{-1}(C)$ define a foliation of    $W\setminus\set{(0,m)}$ 
by nested  Jordan curves  near $(0,m)$. Since $E_0$ is connected, we infer that all curves $\pi_1^{-1}(C)$ sufficiently 
 close to the origin intersect $E$ and 
we are done. 

\medskip

The argument is slightly more delicate when 
$W$ is a graph tangent to $\gamma_m$ at $(0,m)$. Let $k$ be the order of contact between these 
two curves.   By \cite[Lemma 6.4]{bls}, for $w\neq m$ close to $m$, $\gamma_w$ intersects $W$ transversely in exactly $k$ points. 
There exist  local coordinates $(x,y)$ near $(0,m)$ in which $\gamma_m =
\set{y=0}$ and $W = \set{y=x^k}$.  As before, the point is to prove 
that $\pi_1(E\cap W)$ is connected (here the projection $\pi_1$ refers to the new coordinates). By Slodkowski's theorem \cite{slodkowski}, 
we extend the holomorphic motion  to a neighborhood of $E$. Let $\phi$ be the mapping defined near the origin in $\cc$ by 
$$\phi:x \mapsto (x, x^k) \mapsto \mathrm{hol}(x,x^k)\in \set{y=0},$$
where $\mathrm{hol}$ is the holonomy map  sending $(x,y)$ to the central fiber by following the holomorphic motion. 
We want to show that $\phi^{-1}(E_0)$ is connected near the origin. 
The idea is that $\phi$ is a topological deformation of $\phi_0:x\mapsto x^k$. Indeed, by deforming the Beltrami coefficient (ellipse field) 
of the holomorphic motion $E$ to a trivial one (field of circles), we find a continuous
 family of laminations $(E^s)_{s\in [0,1]}$ with $E^1=E$ and $E^0$ is the foliation by horizontal lines. Now, if $y\neq 0$ and 
 $\gamma_{y}^s$ denotes the leaf of $E^s$ through $(0,y)$, by \cite[Lemma 6.4]{bls},  for $s\in [0,1]$ 
 the intersection points 
between $W$ and $\gamma_{y}^s$ move continuously and without collision, thus we can follow them from $s=1$ to $s=0$. 
It follows that  $\phi^{-1}(E_0)$ is   homeomorphic to 
$\phi_0^{-1}(E_0)$ near the origin. Since $E_0$ is connected and contains 0, $\phi_0^{-1}(E_0)$ is connected near the origin, 
and we are done. 

\medskip

The last case is when $W$ is tangent to the vertical axis at $(0,m)$. Then  it is transverse to the lamination, so $W\cap E$ is connected near $(0, m)$ and so does $\pi_1(W\cap E)$.
\end{proof}

\medskip

{\noindent \em Step 3': propagation of laminarity and  conclusion in the laminar case.}
To make the argument easier to understand, we start by giving  a direct proof of  Corollary \ref{cor:non laminar}. 

Assume   that there is an open set $\om\subset N\times \cc$ in which the  Green current $T$ is (nonzero and) laminar, 
that is, it is an integral of compatible holomorphic disks. Let us first observe that these disks cannot all be contained in 
 vertical fibers\footnote{It is   likely that the set of vertical disks subordinate to $T$ is of zero trace measure but we could not prove it.}.  
Indeed otherwise in $\om$ we would have that $T\wedge i dz\wedge d\overline z=0$, thereby  contradicting Lemma \ref{lem:julia}
(see the comments preceding the lemma). 
Thus there exists a disk $U$ in $N$ and a 
non-trivial uniformly laminar current $S\leq T$ made of graphs over $U$. Saturate $U$ under the dynamics of $p$ to obtain 
 $A = \bigcup_{n\geq 0} p^n(U)$ which is typically an annulus (unless $U$ contains the origin).  We now
  show that the existence of such a $S$ forces $J$ to have a uniform laminar structure in $A\times \cc$. 
  
 \begin{lem}\label{lem:annulus}
 There exists a lamination in $A\times \cc$, whose leaves are locally graphs over the first coordinate, such that $J\cap (A\times \cc)$ 
 is a union of leaves. 
 Moreover, every holomorphic disk contained in $J$   must be compatible with this lamination, that is, contained in a leaf.
 \end{lem}
 

\begin{proof}
Let $\el$ be the lamination by graphs over $U$ underlying $S$. By assumption, $\el$ is of positive trace measure. As already explained, the 
image of a graph $\Gamma$ over $U$ is a graph over $p(U)$ and moreover $f : \Gamma\cv f(\Gamma)$ is of multiplicity 1, so if $\Gamma$ 
is a leaf of $\el$, for every $n\geq 1$, $f^n(\Gamma)$ is a graph over $p^n(U)$, contained in $J$ (more precisely, in $\bigcup_{z\in N} J_z$). 
We claim that these graphs are compatible (that is, they are disjoint or coincide). Indeed,  write  
$S$   as $S=\int [\gamma_\alpha]m(d\alpha)$ for some measured  family  of graphs $(\Gamma_\alpha)$, and for notational ease assume 
that $n=1$. Then $f_*S$ expresses as 
  $\int [f(\gamma_\alpha)]m(d\alpha)$, that is, it is a uniformly woven current. Now,  observe that $f_*S\leq f_*T = dT$, hence $f_*S$ is a 
  positive closed current in $p(U)\times \cc$, with continuous potential (see \cite[Lemma 8.2]{bls}) and  
  such that $ (f_*S)^2\leq d^2T^2 = 0$. 
  The intersection of uniformly woven currents is geometric (see e.g. \cite[Prop. 2.6]{ddg2}), so we infer that for a.e. 
  $(\alpha, \beta)$, $f(\Gamma_\alpha)$ and $f(\Gamma_\beta)$ are compatible. To get the result for every $(\alpha, \beta)$, 
  assume that 
  $\alpha_0$ and $\beta_0$ are such that $f(\Gamma_{\alpha_0})$ and $f(\Gamma_{\beta_0})$ have a non-trivial intersection. Then 
  by the 
  persistence of intersections of curves in $\cd$,  for $\alpha$ (resp. $\beta$) close to $\alpha_0$ (resp. $\beta_0$) 
  we     get that $[f(\Gamma_\alpha)]\wedge [f(\Gamma_\beta)]>0$, which contradicts the   generic non-intersection. 
  So we conclude that the leaves 
  of $f(\el)$ are disjoint, as asserted. In particular, $f(\el)$, being a closed family of disjoint graphs over $p(U)$, is a lamination. The same 
  argument shows that for any two integers $n$ and $m$, the leaves of the laminations $f^n(\el)$ and $f^m(\el)$ are compatible, that is, they 
  admit no proper intersections. 
  
  \medskip
  
Reduce $U$ a little bit to obtain an open set $U'$, with corresponding annulus $A'$ and lamination $\el'$.
Consider the set $\bigcup_{n\geq 0} f^n(\el')$. 
This is a uniformly bounded family of compatible graphs over sets of the form $p^n(U')$ in $A'$. 
Notice that $p^n(U')$ is of uniform size and stays at uniform distance from $\fr (p^n(U))$.   
    Recall that over any $p$-invariant circle $C$, 
there is an $f$-invariant ergodic probability measure $\mu_C$, and that $T\wedge idz\wedge d\overline z$ is equivalent to an integral of measures of this form. Since for every $z\in U'$, $0<S\wedge [\set{z}\times \cc] \leq \mu_z$, we get that for every invariant circle in $A'$, 
$\el'$ is of positive $\mu_C$ measure. By ergodicity we infer that for every such $C$, $\bigcup_{n\geq 0} f^n(\el')$ is of full  $\mu_C$ 
measure, and integrating with respect to
 $C$ we deduce  that it is of full $(T\wedge idz\wedge d\overline z)$ measure in $A'$. By Lemma \ref{lem:julia}, 
$\supp(T\wedge idz\wedge d\overline z) = J$ so we conclude  that $\bigcup_{n\geq 0} f^n(\el')$ is dense in $J\cap (A'\times \cc)$.  
Finally, $\overline{\bigcup_{n\geq 0} f^n(\el')}$ defines a lamination of $J\cap 
(A'\times \cc)$. Indeed for every $x\in J\cap (A'\times \cc)$ belongs to the cluster set of a family of graphs $\Gamma'_j$ over 
$p^{n_j}(U')$. Since these graphs extend to $p^{n_j}(U)$ and are uniformly bounded, we obtain a limiting graph through $x$. 
Persistence of isolated intersections implies that any two limiting graphs must be compatible, and we are done. 

\medskip

The second statement of the lemma is obvious: since $J$ is laminated, every disk contained in $J$ and not contained in a leaf would have transverse intersections with nearby leaves, and by using the holonomy we would get that $J$ has non-empty interior, a contradiction. 
\end{proof}

We are now ready to conclude the proof of Corollary \ref{cor:non laminar}. 
It follows from the previous lemma that the holomorphic motion $E$ is compatible with the laminar structure of $J$. 
By Lemma \ref{lem:connex}, 
there is a  component   $V$ of  $ \mathrm{Crit}(f)$ in $\Delta\times \cc$ 
 such that $f(V)$   admits a non trivial intersection with a leaf  of this holomorphic motion  over $A$.  
 Let  $x  \in  V \cap (N\times \cc)$      such that $f(x)=y$ is such an intersection point, and denote by $\Gamma$ 
 the graph of $E$  through $y$. Since such intersections are persistent, 
 shifting $y$ slightly, we may further assume that $ \Gamma$ is not contained in (some other component of) $f(\mathrm{Crit}(f))$. 
 Since $J$ is totally invariant,   $f^{-1}(\Gamma)$ must be  contained in $J$.  Let us show that this is contradictory.

Writing in coordinates $\Gamma = \set{(t,\gamma(t)), t\in U}$ we infer that an equation for $f^{-1}(\Gamma)$ is 
$h(z,w) = q(z,w) - \gamma(p(z)) =0$. Since $p$ has no critical points in $\Delta$ and $x\in \mathrm{Crit}(f)$, 
$\frac{\fr h}{\fr w}  = \frac{\fr q}{\fr w} = 0$ at $x$. If $\frac{\fr h}{\fr z} (x)\neq 0$, this means that $f^{-1}(\Gamma)$ 
has a vertical tangency at $x$, which is impossible by Lemma \ref{lem:annulus}.  Otherwise $\frac{\fr h}{\fr z} (x)= 0$ 
and there are two possibilities: if the equation $(h=0)$ is of multiplicity 1, this implies that  $f^{-1}(\Gamma)$ is singular at $x$, 
which again contradicts Lemma \ref{lem:annulus}. The other option is that 
 $(h=0)$ is smooth with non-trivial multiplicity, but then it must be contained in the critical set, which contradicts our assumptions on 
 $\Gamma$.   Thus in any case, we arrive at a contradiction, and the proof  of Corollary \ref{cor:non laminar} is complete \qed

%
%
%
%
%
%
%
%

\medskip

{\noindent \em Step 3: conclusion in the general  case.}
We now handle the general case, that is we prove Theorem \ref{thm:non quasi laminar}.
We will  design  a different argument for the propagation of laminarity, and obtain a weaker version of Lemma \ref{lem:annulus} which will be 
sufficient for our purposes.  

\medskip

Assume that $T$ is quasi-laminar in some open subset $\om\subset N\times \cc$.  
Recall from \cite{fatou} that  for $\sigma_T$-a.e. $x\in N\times \cc$, there is a unique Fatou direction at $x$,
along which the dynamics is not expanding. This direction coincides with the tangent space to $T$ at $x$, which is well-defined because 
$T$ is simple a.e. on $J_1\setminus J_2$ \cite[Thm 3.4]{fatou}. If $D$ is 
any Fatou disk through $x$, $T_xD$ must then coincide with the tangent space to $T$ at $x$. 

By Lemma \ref{lem:julia}, $T\wedge idz\wedge d\overline z$ is non-zero in $\om$. We claim that for 
$(T\wedge idz\wedge d\overline z)$-a.e. $x$, the tangent space of $T$ at $x$ is not vertical\footnote{Another argument consists in showing that  over every invariant circle,  the Lyapunov exponent of $\mu_C$ in the vertical direction is at least $\frac{\log d}{2}$}. 
Indeed, 
write $T$ as a differential form with measure coefficients 
$T = \sum_{\alpha, \beta \in \set{z,w}} T_{\alpha, \overline \beta}\; id\alpha\wedge d\overline\beta$. Then 
$\sigma_T = T_{z,\overline z} + T_{w, \overline w}$, and by the Radon-Nikodym theorem there exist measurable functions 
$h_{\alpha, \overline\beta}$ such that for $\alpha, \beta \in \set{z,w}$
$T_{\alpha, \overline \beta} = h_{\alpha, \overline\beta} \sigma_T$. 
The matrix 
$(h_{\alpha, \overline \beta})$ has rank 1 $\sigma_T$-a.e. on $J_1\setminus J_2$, and the tangent vector to $T$ is vertical precisely when 
$h_{w, \overline w} = 0$ (which implies $h_{w, \overline z}  = h_{z, \overline w}= 0$). We conclude by observing that since
$T\wedge idz\wedge d\overline z = T_{w,\overline w} = h_{w,\overline w} \sigma_T$, the set $\set{h_{w, \overline w} = 0}$ has zero 
$T\wedge idz\wedge d\overline z$ measure. 

Thus in $\om$ there exists a set of positive trace 
measure of points $x$ such that any Fatou disk at $x$ must have a non-vertical tangent space at $x$. 
For a small value of $r_0$ to be determined shortly, we let $L$ be the set of points $x\in J\cap(N\times \cc)$
%
such that there exists a Fatou disk through $x$, which is a graph over $D(\pi_1(x), 2r_0)$. 
For $r>0$, denote by $\mathcal G(r)$ the set of graphs   over a disk of radius    $r$, that are contained in $\bigcup_{\zeta\in \Delta} K_\zeta$. 
The  compactness of $\Gamma(2r_0)$ implies that  $L$ is a compact subset of $J$  which is of  positive 
$T\wedge idz\wedge d\overline z$-measure for small enough $r_0$. Given such a $r_0$, there exists $0<r\leq r_0$ so that for every $z\in N$, and every $n\geq 0$, $p^n(D(z, 2r_0))$ contains $D(p^n(z), 2r)$. 

Recall from \cite{jonsson} that $z\mapsto J_z$ is lower-semicontinuous. In particular $\delta =  \min_{z\in \overline N} \mathrm{diam}(J_z)$ is a 
well defined positive quantity.  By compactness of the set of graphs, if $r$ is small enough and 
$\gamma$ is a graph over some disk $D(z,2r)$
that is  contained in $\bigcup_{\zeta\in \Delta} K_\zeta$, 
then $\mathrm{diam}(\gamma\rest{D(z,r)})< \frac{\delta}{4}$. We fix $r_0$ such that the associated $r$ has this
 property.  
 
 \medskip
 
To prove uniqueness and disjointness properties for the Fatou disks, we will use an argument based on an expansivity property of the 
fiberwise dynamics, which itself follows from a form of   mixing in the fiber direction. The statement we need  is the following. 

\begin{lem}\label{lem:expansive}
Let as above  $\delta =  \min_{z\in \overline N} \mathrm{diam}(J_z)$. Then for every $z\in \overline{N}$, and 
every subset $A\subset J_z$ of positive $\mu_z$-measure, there exists $n\geq 0$ such that $\mathrm{diam}(f^n(A))>\frac{3\delta}{4}$. 
\end{lem}
 
 \begin{proof}
 Fix $z\in \overline N$ and put $z_n = p^n(z)$. 
 We first claim that if $\varphi\in L^2(\mu_z)$ and $\psi$ is any  test function in $\Delta\times \cc$, then
 \begin{equation}\label{eq:mixing}
 \abs{\int (\psi\circ f^n) \varphi \; \mu_z - \int \psi \; \mu_{z_n} \int \varphi \; \mu_z} \underset{n\cv\infty}{\longrightarrow}0.
 \end{equation}
Indeed observe first that by approximating in $L^2(\mu_z)$, 
it is enough to establish this for a smooth $\varphi$. 
Using the relation $f_*(h \mu_z)  = \unsur{d} (f_*h) f_*\mu_z$ for the slice measures, we rewrite 
$$\int (\psi\circ f^n) \varphi \; \mu_z = \int \lrpar{\unsur{d^n} (f^n)_*\varphi} \psi \; \mu_{z_n},$$
so that \eqref{eq:mixing} boils down to 
 \begin{equation}\label{eq:mixing bis}
\abs{\int \lrpar{\unsur{d^n} (f^n)_*\varphi -\int \varphi \; \mu_z } \psi \; \mu_{z_n} }
\underset{n\cv\infty}{\longrightarrow}0.
\end{equation}
Now for smooth $\varphi$, \cite[Prop. 2.6]{jonsson2} asserts that 
$$\mu_{z_n} \lrpar{\set{w \in \set{z_n} \times \cc, \   \abs{\unsur{d^n} (f^n)_*\varphi (w) - \int \varphi \; \mu_z } >t
 }}\leq \frac{C\norm{\varphi}_{C^2}}{td^n},$$ from which \eqref{eq:mixing bis}, hence \eqref{eq:mixing}, readily follows.
 
 \medskip
 
 Fix $z_0$ such that $\mathrm{diam}( J_{z_0})> \frac{3\delta}{4}$. Fix   $x_0$ and $ x'_0$ in $J_{z_0}$ at  
distance greater than $ \frac{3\delta}{4}$ from each other, and $\psi_0$ and $\psi'_0$ be two bump functions (in $\Delta\times \cc$) supported in   small respective neighborhoods of $x_0$ and $x'_0$. The  continuity of $z\mapsto \mu_z$ implies that if $  z$ 
is close to $z_0$, the integrals
$\int \psi_0  \; \mu_{  z}$ and $\int \psi'_0  \; \mu_{  z}$ are positive. 

To conclude the proof of the lemma, we fix a subsequence $n_j$ such that $z_{n_j}\cv z_0$, and  
set $\varphi = \mathbf{1}_A$. For $\psi = \psi_0$ or $\psi =\psi'_0$, from  \eqref{eq:mixing} we get 
that for large $j$, $\int_A \psi\circ f^{n_j} d\mu_z >0$. This  implies that $f^{n_j}(A)$   intersects both $\supp(\psi_0)$ and 
$\supp(\psi'_0)$, and the result follows.
   \end{proof}
 
 The basic mechanism deriving laminarity from Lemma \ref{lem:expansive}   is contained in the following lemma. 
 
 \begin{lem}\label{lem:disjoint}
Let $z\in \overline N$ and  assume that there exists a compact set  $A\subset J_z$ of 
 positive $\mu_z$-measure, together with  a measurable family of Fatou disks  
 $\set{\gamma_x, \ x\in A}$, that   are graphs over $D(z,2r)$. 
 
 Then the graphs $\gamma_x$ are disjoint over $D(z,r)$. 
 \end{lem}
 
 \begin{proof}
 Observe first  that  
 if $\gamma$ is a Fatou graph over $D(z,2r)$
 such that $\gamma(z)\in J_z$, then  by normality of the iterates, $f^n(\gamma$ is uniformly bounded, hence 
  $\gamma \subset  \bigcup_{\zeta\in \Delta} K_\zeta$.
Introduce the space $\mathcal{G}_z(2r)$ of graphs over $D(z,2r)$ that are contained in  $\bigcup_{\zeta\in \Delta} K_\zeta$, endowed with the compact-open topology,
 which is a compact metrizable space.
%
  In particular we can project the 
 measure $\mu_z$ by     $e : x\mapsto \gamma_x$. 
 This mapping is injective since   $\gamma_{(z,w)}(z) = w$, so we infer that the image measure $e_*(\mu_z)$
 has no atoms. 
 
 Now assume that there exists $x\neq y \in J_z$ such that $\gamma_x$ and $\gamma_{y}$ intersect over $D(z,r)$.   
 Since $e_*(\mu_z)$  is diffuse, by the continuity of isolated intersections,  there exists a set $A'\subset A$ of positive $\mu_z$ measure  such that if $y_1\in A'$ 
 then $\gamma_{y_1}$ and $\gamma_x$ intersect over $D(z,r)$. Therefore for $y_1,y_2\in A'$ there is a chain $(\gamma_{y_1}, \gamma_x, \gamma_{y_2})$ of three intersecting 
 Fatou graphs connecting $y_1$ and $y_2$. Now recall that by definition of $r$, if $\gamma$ is one of these graphs, then for every 
 $n$, 
 $\mathrm{diam}(f^n(\gamma))$ is smaller than $\frac{\delta}{4}$.
 From this 
 we deduce that for all $n\geq 0$, $\dist (f^n(y_1), f^n(y_2)) < \frac{3\delta}{4}$, which contradicts Lemma \ref{lem:expansive} and 
 thereby completes the proof.  
\end{proof}

  Recall our assumption that there exists a set $L$ of positive trace measure of points $x=(z,w)$ admitting a Fatou disk through $x$ 
  that  is a graph over $D(z, 2r_0)$. By 
  the slicing formula, $T\wedge idz\wedge d\overline z$ is an integral of $\mu_z$, hence  
 there exists a set of positive Lebesgue measure of   of invariant circles $C$ 
 such that  $\mu_C(L)>0$. 
 
\begin{lem}\label{lem:unique}
Let $C$ as above be an invariant circle 
  such that $\mu_C(L)>0$. Let  $z\in C$ and      $x=(z, w)\in J_z$.
Then  there exists  a unique Fatou disk $\gamma_x$
 through  $x$ which is a graph over $D(z,2r)$. In addition, if $x, x'\in J_z$, are distinct points, then $\gamma_x$ and $\gamma_{x'}$ are disjoint over 
 $D(z,r)$.  
\end{lem}

\begin{proof}
For the existence, we argue as in Lemma \ref{lem:annulus}. Indeed, by  the ergodicity of $\mu_C$ and the
definition of $r$, we infer that for $\mu_C$-a.e. $x = (z,w)$, there exists a Fatou disk through $x$
 that is a graph over $D(z,2r)$. By the 
 compactness  of the space of graphs, we extend this family to its closure, thus obtaining 
  for {every} $z\in C$ and every $w\in J_z$  a Fatou  graph   
 over $D(z,2r)$ containing $(z,w)$. 

\medskip

For the uniqueness, fix a Fatou graph $\gamma$ over $D(z, 2r)$ with $\gamma(z)  = w$. 
Recall that the measure 
$\mu_z\rest{L\cap (\set{z}\times \cc)}$ is diffuse, so its support has no isolated points. 
 For every $x\in J_z$, let $\mathcal{D}_x\subset \mathcal{G}_z(2r)$ be the set of Fatou graphs over $D(z,2r)$ 
 through $x$, which by assumption is non-empty. Being bounded and closed, $\mathcal{D}_x$ is compact.  
Put $\mathcal{D} = \bigcup_{x\in J_z} \mathcal{D}_x$, which is also compact.  

Consider the natural continuous map $\phi: \mathcal D \cv J_z$ defined by $\gamma\mapsto \gamma(z)$. The  
``measurable axiom of choice" (see \cite[Thm. 6.9.7]{bogachev} for the version that we use here) implies that this map admits 
 a Borel section, that is, an injective  Borel map $e$ such that  $\phi\circ e = \mathrm{Id}$. Without changing these properties we may 
 assume that $e(x)=\gamma$.
%
%
%
  By Lemma \ref{lem:disjoint}, the graphs $e(x)$, $x\in J_z$ are disjoint.  Now let $(x_n)$ be any sequence in $J_z$
  converging to $x$. Since for every $n$, $e(x_n)\cap e(x) = \emptyset$, by the Hurwitz theorem any cluster 
  limit of $e(x_n)$  must coincide with $e(x)= \gamma$. If $\gamma'$ is any other Fatou graph at $x$, we can freely modify $e$  
  so that $e(x) = \gamma'$, so the limit of $e(x_n)$ also equals $\gamma'$ and we conclude that $\gamma$ is unique, as desired. 
  
  The disjointness assertion of the lemma now follows directly from Lemma \ref{lem:disjoint}
\end{proof} 
 
%
%
%

 It follows from the previous lemma and the  $\Lambda$-lemma of Mañé, Sad, and Sullivan \cite{mss} that 
  $\bigcup_{x\in J_z} \gamma_x$ is a lamination in $D(z,r)\times \cc$. Likewise, a similar argument shows that 
 $\bigcup_{z\in C, x\in J_z} \gamma_x$ 
 forms a lamination over some neighborhood of $C$ (we will not need this result), however we cannot characterize its support nor
  even  show that it is contained in $J$. 
  
  Let us now conclude the proof  of Theorem \ref{thm:non quasi laminar} 
 similarly to   the laminar case (Step 3'). 
 Indeed by Lemma \ref{lem:connex}, 
there exists   $z\in C$ and  a  component   $V$ of  $ \mathrm{Crit}(f)$  through $x =(z,w)$ 
 such that $f(V)$   admits a non trivial intersection with a leaf $\Gamma$  of $E$ over $C$ (recall that $E$ is the holomorphic motion issued from the hyperbolic set $E_0$ on the central fiber). 
 Shifting $C$ slightly if needed, we may assume that 
 $\Gamma$ is not contained in $f( \mathrm{Crit}(f))$. Now consider the subvariety   $f^{-1}(\Gamma)$  near  $x$, which 
due to this     assumption  is of multiplicity 1. Looking at the equation of $f^{-1}(\Gamma)$ as in the laminar case, we conclude that there are two 
possibilities: either $f^{-1}(\Gamma)$ is smooth with a vertical tangency at $x$, or it is singular at $x$. 
In any case, $f^{-1}(\Gamma)$ admits an isolated intersection with $\gamma_x$ so 
by the lamination structure of the set of Fatou disks, 
there is a set of positive $\mu_z$ measure of $x'=(z,w')$ such that 
$f^{-1}(\Gamma)$ intersects  the graph $\gamma_{x'}$. 
Since $\Gamma$ is a Fatou graph,  the iterates 
 $f^n\rest{f^{-1}\Gamma}$, $n\geq 0$,  form a normal family. 
 Thus, arguing exactly as in Lemma \ref{lem:disjoint} we construct a set of positive 
measure $A \subset J_z$    such that if $y_1$, $y_2$ in $A$, $\dist(f^n(y_1), f^n(y_2))< \frac{3\delta}{4}$. 
This contradiction finishes the proof.  
\qed


\section{Further results and comments} \label{sec:further}

\subsection{Higher dimension} The laminarity problem makes sense in higher dimension as well. 
Given a holomorphic endomorphism of $\pk$ of degree $d$  with  Green current $T$, 
we define its $q^{\rm th}$ Julia set by 
$J_q = \supp(T^q)$. Then it may be asked whether for $1\leq q\leq k-1$,
$J_q\setminus J_{q+1}$ can be filled with Fatou disks of codimension $q$. We obtain a negative answer to this question by simply taking  the direct  product of the example $f$ of Theorem \ref{thm:non quasi laminar}
  with a polynomial map on $\cc^{k-2}$, for instance the monomial map 
$$g: (z_1,\ldots , z_{k-2}) \longmapsto (z_1^d,\ldots , z_{k-2}^d).$$ 
The  product map $F = (f,g)$ is a polynomial mapping on $\cc^k$ which extends to $\pp^k$ and its Green current is
 given by $T_F  = \pi_1^*T_f + \pi_2^* T_g$, where 
$$\pi_1:\cd\times \cc^{k-2}\cv \cd \text{ and } \pi_2:\cd\times \cc^{k-2}\cv \cc^{k-2}$$ are the natural projections. 

Let $\om\subset \cc^{k-2}$ be an open set such that $\om\cap J_{q-1}(g)\neq \emptyset$ and $\om\cap J_{q}(g)= \emptyset$. 
This happens for instance if in $\om$, $k-1-q$ coordinates have modulus smaller than 1 and the $q-1$ remaining 
 ones cross the unit circle.
 Then in  $\om$, $J_{q-1}(g)$ is foliated by Fatou disks of dimension $k-1-q$.  
 Now for the product map $F$, in $N\times \cc\times \om$ we have that 
 $J_q(F)  = J(f)\times J_{q-1}(g)$ and $J_{q+1}(F) = \emptyset$, and it is easy to see    that 
 $T_F^q =( \pi_1^*T_f) \wedge ( \pi_2^* T_g^{q-1})$ is not laminated 
 by Fatou disks of dimension $k-q$ (i.e. of codimension $q$).

\subsection{Another question of Forn\ae ss and Sibony}
In \cite{fs questions},  Forn\ae ss and Sibony ask the following question: if $f$ is an endomorphism of $\pk$,
 is it true that  the non-wandering set the 
closure of the union of the set of periodic points and of the set {\em Siegel varieties}? A Siegel variety is a (local) 
analytic subset $X$ such that  there exists  a subsequence $n_j$ such that $f^{n_j}\rest{X}$ converges to the identity. 
Notice that in the setting of our main theorem, $J\cap (N\times \cc)$ is contained in the non-wandering set. Indeed, since the dynamics is 
ergodic over every invariant circle, $(T\wedge idz\wedge d\overline z)$-almost every point is recurrent. 
Without solving this problem, Theorem \ref{thm:non quasi laminar} says at least that in $N\times \cc$, Siegel varieties 
 must be very small. For 
instance, the set of Siegel varieties which are graphs over    a given  invariant sub-annulus in $N$ is a nowhere dense set in $J$ 
of trace measure zero. Indeed, otherwise this would give rise to 
a set of positive trace measure of Fatou graphs, which is impossible, as the proof of Theorem \ref{thm:non quasi laminar} shows.

\end{document}